\documentclass[12pt]{amsart}
\usepackage[utf8]{inputenc}
\usepackage{tikz}
\usetikzlibrary{shapes.geometric}
\usetikzlibrary{arrows.meta}
\usetikzlibrary{patterns,snakes}

\DeclareFontFamily{U}{BOONDOX-calo}{\skewchar\font=45 }
\DeclareFontShape{U}{BOONDOX-calo}{m}{n}{
  <-> s*[1.05] BOONDOX-r-calo}{}
\DeclareFontShape{U}{BOONDOX-calo}{b}{n}{
  <-> s*[1.05] BOONDOX-b-calo}{}
\DeclareMathAlphabet{\mathcalboondox}{U}{BOONDOX-calo}{m}{n}
\SetMathAlphabet{\mathcalboondox}{bold}{U}{BOONDOX-calo}{b}{n}
\DeclareMathAlphabet{\mathbcalboondox}{U}{BOONDOX-calo}{b}{n}

\usepackage{epsf}
\usepackage{todonotes}
\usepackage{verbatim}
\usepackage{amsmath}
\usepackage{amsopn}
\usepackage{amssymb}
\usepackage{latexsym}
\usepackage{amsfonts}
\usepackage{amsthm}
\usepackage{tikz}

\newcommand{\Des}{{\rm{Des}}}
\newcommand{\des}{{\rm{des}}}

\newtheorem{thm}{Theorem}[section]
\newtheorem{pro}[thm]{Proposition}
\newtheorem{lem}[thm]{Lemma}
\newtheorem{cla}[thm]{Claim}

\newtheorem{cor}[thm]{Corollary}

\theoremstyle{definition}
\newtheorem{obs}[thm]{Observation}

\newtheorem{exa}[thm]{Example}
\newtheorem{example}[thm]{Example}
\newtheorem{defn}[thm]{Definition}

\newtheorem{algo}[thm]{Algorithm}

\newtheorem{prop}[thm]{Proposition}
\newtheorem{conj}[thm]{Conjecture}

\newcommand{\een}{\end{enumerate}}
\newcommand{\blem}{\begin{lem}}
\newcommand{\elem}{\end{lem}}
\newcommand{\bcl}{\begin{cla}}
\newcommand{\ecl}{\end{cla}}
\newcommand{\ethm}{\end{thm}}
\newcommand{\bpr}{\begin{pro}}
\newcommand{\epr}{\end{pro}}
\newcommand{\bco}{\begin{cor}}
\newcommand{\eco}{\end{cor}}
\newcommand{\bcon}{\begin{conj}}
\newcommand{\econ}{\end{conj}}
\newcommand{\bde}{\begin{defn}}
\newcommand{\ede}{\end{defn}}
\newcommand{\bex}{\begin{exa}}
\newcommand{\eexa}{\end{exa}}
\newcommand{\bobs}{\begin{obs}}
\newcommand{\eobs}{\end{obs}}
\newcommand{\bexe}{\begin{exe}}
\newcommand{\eexe}{\end{exe}}

\begin{document}

\title[Worpitzky identity for Coxeter groups of types $B$ and $D$]{The Worpitzky identity for the groups of signed and even-signed permutations}

\author{Eli Bagno, David Garber and Mordechai Novick}

\address{Eli Bagno and Mordechai Novik, Jerusalem College of Technology\\
21 Havaad Haleumi St. Jerusalem, Israel}
\email{bagnoe@jct.ac.il,mnovick@jct.ac.il}

\address{David Garber \\
Department of  Applied Mathematics, Holon Institute of Technology,
52 Golomb St., PO Box 305, 58102 Holon, Israel}
\email{garber@hit.ac.il}

\begin{abstract}
The well-known Worpitzky identity
provides a connection between two bases of $\mathbb{Q}[x]$: The standard basis $(x+1)^n$ and the binomial basis ${{x+n-i} \choose {n}}$, where the Eulerian numbers for the Coxeter group of type $A$ (the symmetric group) serve as the entries of the transformation matrix.
Brenti has generalized this identity to the Coxeter groups of types $B$ and $D$ (signed and even-signed permutations groups, respectively)
using generating function techniques.

Motivated by Foata-Sch\"utzenberger and Rawlings' proof for the Worpitzky identity in the symmetric group, we provide combinatorial proofs of this identity and for their $q-$analogues in the Coxeter groups of types $B$ and $D$.
\end{abstract}

\maketitle

\section{Introduction}

The well-known Worpitzky identity involves the Eulerian numbers, the original definition of which was given by Euler in an analytic context \cite[\S 13]{Eu}. Later, these numbers began to appear in combinatorial problems, and this is the context we choose to present them here.

Let $S_n$ be the symmetric group on $n$ elements. For any permutation $\pi\in S_n$, we say that $\pi$ has a {\em descent} at position $i$ if $\pi(i)>\pi(i+1)$, and we denote by ${\rm Des}(\pi)$ the set of descents:
\begin{equation}\label{des_typeA}
{\rm Des}(\pi):=\{i \in [n-1]\mid \pi(i)>\pi(i+1)\}.
\end{equation}
We denote the number of descents in $\pi$ by ${\des}(\pi):=|\Des(\pi)|$.

The Eulerian number $A(n,k)$ counts the number of permutations in $S_n$ having $k$ descents:
$$A_{n,k}=|\{\pi \in S_n : {\rm des}(\pi)=k\}|.$$

One of the celebrated identities involving Eulerian numbers is the Worpitzky identity:
\begin{equation}\label{Worpitzky_A}
(k+1)^n = \sum\limits_{i=0}^{n-1} A_{n,i} {{k+n-i} \choose {n}}.
\end{equation}

An excellent overview of the Worpitzky identity and the Eulerian numbers can be found in Petersen's book \cite[Chap. 1]{petersen}.

Worpitzky's identity was generalized to the Coxeter groups of types $B$ and $D$ in Borowiec and M{\l}otkowski \cite{BoMl}, though they used a different set of Eulerian numbers.

Generalizations of the Worpitzky identity, using the algebraic definition of the descents in these groups, were introduced by Brenti \cite[Theorem 3.4(iii) for $q=1$ and Corollary 4.11]{Brenti}:

$$(2x+1)^{n}=\sum_{k=0}^n{{x+n-k} \choose {n}}B_{n,k} \qquad ({\rm type}\ B),$$

$$(2x+1)^{n}-2^{n-1}(\mathcalboondox{B}_n(x+1)-\mathcalboondox{B}(n))=\sum_{k=0}^n{{x+n-k} \choose {n}}D_{n,k}\qquad ({\rm type}\ D),$$
where $B_{n,k}$ and $D_{n,k}$ are the Eulerian numbers of types $B$ and $D$ respectively (these notations will be defined in the next section), $\mathcalboondox{B}(n)$ is the $n$-th Bernoulli number and $\mathcalboondox{B}_n(x)$ is the $n$-th Bernoulli polynomial (see \cite{Mezo} for the definitions of these concepts).
\medskip

Combinatorial identities usually have more than one possible proof. Some of them are analytic, some algebraic in nature, but the most beautiful ones are combinatorial, meaning that both sides of the identity count the same set of elements in different ways.

In our context, Foata and Sch\"utzenberger \cite[p. 40]{FoSc} have proved the Worpitzky identity for the Coxeter group of type $A$ in a combinatorial way (see also Rawlings \cite{Rawlings} and Petersen \cite[p. 366]{petersen}).
On the other hand, Brenti's proofs for the generalizations of Worpitzky's identities are non-combinatorial, and use generating function techniques.

\medskip

Our contribution in this paper is combinatorial proofs for the $q$-analogues of the Worpitzky identity for types $B$ and $D$:
$$(1+(1+q)m)^n=\sum\limits_{k=0}^n{{n+m-k}\choose {n}}B_{n,k}(q) \qquad ({\rm type}\ B)$$
{\small$$\begin{array}{l}(1+2m)((1+q)m)^{n-1}-(1+q)^{n-1}(\mathcalboondox{B}_n(m+1)-\mathcalboondox{B}(n))=\\  \qquad\qquad\qquad\qquad\qquad\qquad =\sum\limits_{k=0}^n{{n+m-k}\choose {n}}D_{n,k}(q)\end{array} \quad ({\rm type}\ D)$$}

These $q$-analogues appear in Brenti \cite{Brenti} (the identity for type $B$ is Theorem 3.4(iii) and the identity for type $D$ is referred to implicitly before Theorem 4.10).

Our combinatorial proofs are in the same spirit of the proof of Foata-\break Sch\"utzenberger \cite{FoSc} for type $A$.  In both types, we count vectors of length $n$ over a certain set, and in the case of type $D$ we encounter a problem of missing vectors, which is corrected by adding a term of the form of a Bernoulli polynomial. This phenomenon appears regarding other statistics in Coxeter groups of type $D$, see e.g. \cite{BBG}.

The proof of the identity for type $D$ integrates direct combinatorial arguments together with formal manipulations of algebraic expressions, sometimes referred to as ``manipulatorics'', see \cite[p. 10]{petersen}. This approach is, in general and in our case as well, less satisfying, but is often necessary to reformulate a complex bijective argument as an elegant formula.

\medskip

The paper is organized as follows. In Section \ref{prelim} we give some preliminaries, including the definitions of the Coxeter groups of types $B$ and $D$ and the Eulerian numbers associated with them.
Sections \ref{section type B} and \ref{section type D} present the combinatorial proofs of the identities for types $B$ and $D$ respectively.

\section{Preliminaries and definitions}\label{prelim}

In this section, we provide some background on the Coxeter groups of types $B$ and $D$, and on the sets of vectors which we will count in our proofs of the Worpitzky identities for these groups. A general reference is Chapter 8 of Bjorner-Brenti's book \cite{Bjorner-Brenti}.

\subsection{The Coxeter group of type B}
Define $B_n$ as the group of \textit{signed permutations} on $\{1,\dots,n\}$, i.e., the set of permutations $\pi$ on $\{\pm 1, \pm 2,\dots,\pm n\}$ such that $\pi(-i)=-\pi(i)$ for $1 \leq i \leq n$. We consent that $\pi(0)=0$ and occasionally  write $\pi_i$ instead of $\pi(i)$.
This is the standard combinatorial realization of the Coxeter groups of type $B$.

We define some statistics on the group $B_n$.
First, for a permutation $\pi\in B_n$, define: $$\Des_A(\pi) = \{i:\pi(i)>\pi(i+1),1\leq i\leq n-1\},$$
and then denote: $\des_A(\pi)=\lvert\Des_A(\pi)\rvert.$

\medskip

Now, for $\pi\in B_n$, we define:
$$\Des_B(\pi)=\left\{\begin{array}{cl} \Des_A(\pi)\cup\{0\} & \pi(1)<0 \\ \Des_A(\pi) & \pi(1)>0 \end{array}\right.$$
As before, we denote: $\des_B(\pi)=\lvert\Des_B(\pi)\rvert$.
\medskip
\begin{example}
Let $\pi=[\bar{1}2\bar{5}43]$.
Then ${\rm Des}_B(\pi)=\{0,2,4\}$
\end{example}

Let $B_{n,k}=\lvert\{\pi\in B_n:\des_B(\pi)=k\}\rvert$.
The number $B_{n,k}$ is called the {\it Eulerian number of type $B$}. These numbers constitute the sequence A060187 in OEIS \cite{OEIS}.

We define also another statistic:
$${\rm neg}(\pi)=|\{i : \pi(i)<0, 1\leq i \leq n\}|,$$ and the $q$-analogue of $B_{n,k}$ is:
$$B_{n,k}(q)=\sum\limits_{{\tiny\begin{array}{c} \pi\in B_n\\ \des_B(\pi)=k \end{array}}}q^{{\rm neg}(\pi)}.$$

\begin{example}
Let $\pi=[\bar{1}2\bar{5}43]$.
Then ${\rm neg}(\pi)=2$.
\end{example}

\subsection{The Coxeter group of type D}

Denote by $D_n$ the group of signed permutations on $\{1,\dots,n\}$ with an even number of negative elements. This is the standard combinatorial realization of the Coxeter group of type $D$.

Before presenting the Eulerian numbers  for $D_n$, we need the following definitions:
For $\pi\in D_n$, define: $$\Des_D(\pi)=\left\{\begin{array}{cl} \Des_A(\pi)\cup\{0\} &  \pi(1)+\pi(2)<0 \\ \Des_A(\pi) & \pi(1)+\pi(2)>0 \end{array}\right.$$
and denote: $\des_D(\pi)=\lvert\Des_D(\pi)\rvert$.

\begin{example}
Let $\pi=[\bar{3}26\bar{5}14]$.
Then: ${\rm Des}_D(\pi)=\{0,3\}$ and\break ${\rm des}_D(\pi)=2$.
\end{example}

Let $D_{n,k}=|\{\pi \in D_n : {\rm des}_D(\pi)=k\}|$ be the {\it Eulerian number of type $D$} (sequence A262226 in OEIS \cite{OEIS}). For the $q-$analogue,
let $$D_{n,k}(q)=\sum\limits_{{\tiny\begin{array}{c} \pi\in D_n\\ \des_D(\pi)=k \end{array}}}q^{{\rm neg}_2(\pi)},$$
where $${\rm neg}_2(\pi)=|\{i\in \{2,\dots,n\} \mid \pi(i)<0\}|.$$
In the last example, ${\rm neg}_2(\pi)=1$.

\subsection{Definitions for vectors}
Define on the alphabet
$$\Sigma=\{ {\bf 0}, \pm 1, \pm 2, \dots , \pm m \},$$
the following linear order (which will henceforth be referred to simply as the ``defined order"):
$$ {\bf 0} \prec -1 \prec 1 \prec -2 \prec 2 \prec \dots \prec -m \prec m.$$

In the context of the vectors, we define two different versions of the parameter neg, which counts the number of negative elements, one for type $B$ and the other for type $D$.

\begin{defn}
The parameter ${\rm neg}(\vec{v})$ is defined to be the number of negative entries in $\vec{v}$.

The parameter ${\rm neg}_2(\vec{v})$ is defined to be the number of negative entries in $\vec{v}$, excluding the smallest element of $\vec{v}$ with respect to the order defined above  (note the difference between this definition and the definition of ${\rm neg}_2$ for a permutation).
\end{defn}

\section{The Worpitzky identity for type $B$}\label{section type B}
For the Coxeter group $B_n$, the following identity was proven by Brenti \cite[Theorem 3.4(iii)]{Brenti}:
\begin{thm}\label{worp type B}
$$(1+(1+q)m)^n=\sum\limits_{k=0}^n{{n+m-k}\choose {n}}B_{n,k}(q).$$
\end{thm}

The proof of this identity in \cite{Brenti} uses manipulations of generating functions (see the proof of Theorem 3.4(ii)). Here, we present a combinatorial proof based on a direct counting argument.

\medskip

We start by presenting an algorithm which associates with each vector a signed permutation.

\begin{algo}\label{B_n algorithm}
Let $\vec{v}=(a_1,a_2,\dots,a_n)\in (\{{\bf 0},\pm1, \dots, \pm m\})^n $. First write the entries of $\vec{v}$ in a row according to the defined order. This yields a permutation $\pi=[\pi_1,\dots,\pi_n] \in S_n$ satisfying $a_{\pi_1}\leq a_{\pi_2} \leq \cdots \leq a_{\pi_n}$.
Moreover, if we have $a_{\pi_i}=a_{\pi_{i+1}}$, we construct the permutation $\pi$ so the following condition holds:
\begin{equation}\label{des condition}
\mbox{If } a_{\pi_i} \geq {\bf 0} \mbox{, then } \pi_i<\pi_{i+1} \mbox{, and } \pi_i>\pi_{i+1} \mbox{ otherwise.}
\end{equation}
This means that $\pi$ reads equal entries of $\vec{v}$ from left to right if they are nonnegative, and from right to left otherwise.

\medskip

For example, let $m=3$, $n=6$, and $\vec{v}=(1,-2,{\bf 0},-1,3,-2)$.  In the defined order we have: $a_3\prec a_4\prec a_1\prec a_2 = a_6 \prec a_5$, and hence (since $a_2=a_6<0$) we have $\pi=[3,4,1,6,2,5]$.

\medskip

Finally, define $\sigma \in B_n$ in the following way: For each $i$ satisfying $a_{\pi_i} \geq {\bf 0}$, define $\sigma_i=\pi_i$ and for each $i$ satisfying $a_{\pi_i}<{\bf 0}$, define $\sigma_i=-\pi_i$.  In the example above, we have $\sigma=[3,-4,1,-6,-2,5]$.

Note that $\pi_i=|\sigma_i|$. \end{algo}

Furthermore, if $|a_{\pi_j}| = |a_{\pi_{j+1}}|$, we have $\sigma_j < \sigma_{j+1}$. Indeed:
\begin{itemize}
\item If $a_{\pi_j} = a_{\pi_{j+1}}>0$, by Condition (\ref{des condition}) above, we have $\pi_j < \pi_{j+1}$.
\item If $a_{\pi_j} = a_{\pi_{j+1}}<0$, by Condition (\ref{des condition}) we have $\pi_j > \pi_{j+1}$, and thus $\sigma_j=-\pi_j < -\pi_{j+1}=\sigma_{j+1}$.
\item If $a_{\pi_j} =-a_{\pi_{j+1}}$, according to the defined order, we must have $a_{\pi_j}<0$, $a_{\pi_{j+1}}>0$, and hence $\sigma_j <0<\sigma_{j+1}$.
\end{itemize}
Therefore, by the definition of the descents of type $B$ and by the construction of $\pi$, we conclude the following:
\begin{equation}\label{des2 condition}
\mbox{If } j \in {\rm Des}_B(\sigma) \mbox{, then } |a_{\pi_j}| < |a_{\pi_{j+1}}| \end{equation}
(while assuming $\pi_0=0$, $a_{\pi_0}=a_0=0$ and recall that $\pi_i=|\sigma_i|$).

\begin{proof}[Proof of Theorem \ref{worp type B}]
The left hand side counts the number of vectors of the form $$\vec{v}=(a_1,a_2,\dots,a_n) \in (\{0,\pm1, \dots, \pm m\})^n,$$
where each vector $\vec{v}$ contributes $q^{{\rm neg}(\vec{v})}$.%to the power of ${\rm neg}(\vec{v})$. %

As we show below, the right hand side counts the same set of vectors, where they are classified by signed permutations.

\medskip

Denote by $\phi_{n,m}$ the mapping $\vec{v} \mapsto \sigma$ defined in Algorithm \ref{B_n algorithm} and note that ${\rm neg}(\vec{v})={\rm neg}(\phi_{n,m}(\vec{v}))$.

We show that the number of vectors associated by the algorithm to a given permutation $\sigma\in B_n$ is exactly ${{n+m-\des_B (\sigma)}\choose{n}}$, i.e.,
$$\left| \phi_{n,m}^{-1}(\sigma) \right| = {{n+m-\des_B(\sigma)}\choose{n}},$$
from which the theorem immediately follows.

\medskip

We start with an example: let $\sigma=[2,-1,4,-5,3]\in B_5$ and $m=3$.  Note that
${\rm Des}_B(\sigma)=\{1,3\}$.

We have to find the vectors $$\vec{v}=(a_1,a_2,a_3,a_4,a_5) \in (\{0, \pm 1 , \pm 2 , \pm 3 \})^5$$
satisfying $a_1<0, a_5<0$ and $a_2 \geq 0, a_3>0, a_4>0$ (in the usual integer order) and
\begin{equation}\label {abs values}
0 \leq |a_2| < |a_1| \leq |a_4| < |a_5| \leq |a_3| \leq 3,
\end{equation}
and we have to show that there are ${{5+3-2}\choose{5}} =6$ such vectors.

The sequence of inequalities (\ref{abs values}) is equivalent in turn to
\begin{equation}
1 \leq \underbrace{|a_2|+1}_{=b_1} < \underbrace{|a_1|+1}_{=b_2} < \underbrace{|a_4|+2}_{=b_3} < \underbrace{|a_5|+2}_{=b_4} < \underbrace{|a_3|+3}_{=b_5} \leq 6,
\end{equation}
so we can conclude that the number of vectors satisfying the sequence of inequalities (\ref{abs values}) is ${6 \choose 5}=6$
as claimed.

\medskip

Here is the argument in general:
Let $\sigma=[\sigma_1,\dots,\sigma_n] \in B_n$.  We have to find the number
of vectors
$$\vec{v}=(a_1,a_2,\dots,a_n)\in (\{0,\pm1,\dots,\pm m\})^n$$
such that for each $j \in \{1,\dots,n\}$ satisfying $\sigma_j<0$, one has $a_{|\sigma_j|}<0$ and:
$$0 \leq |a_{|\sigma_1|}|\leq |a_{|\sigma_2|}| \leq \cdots \leq |a_{|\sigma_n|}|\leq m,$$
with the property that the $i^{\rm th}$ order sign in this sequence of inequalities is strict if $i \in {\rm Des}_B(\sigma)$, for $0 \leq i \leq n-1$.

By adding $1$ to each term in this sequence of inequalities, we obtain:
$$1 \leq |a_{|\sigma_1|}|+1\leq |a_{|\sigma_2|}|+1 \leq \cdots \leq |a_{|\sigma_n|}|+1\leq m+1.$$
Now, in order to convert to strict order signs, we add $1$ to the right hand side of each non-strict inequality (and to each inequality to the right of it).
Since the number of strict order signs in the original sequence of inequalities is ${\rm des}_B(\sigma)$, at the end of this process, we have:
$$1 \leq b_1 < b_2< \cdots< b_n \leq m+n-{\rm des}_B(\sigma),$$
where $b_i=\left| a_{|\sigma_i|} \right|+\left|\left\{j \in {\rm Des}_B(\sigma) \mid j<i \right\}\right| +1$.

The number of integer solutions of this sequence of inequalities is:\break $\tiny{\left(\hspace{-4pt} \begin{array}{c} m+n-{\rm des}_B(\sigma) \\ n \end{array} \hspace{-4pt} \right)}$.
Note that for each $i$, after fixing the value of $b_i$, the value of $a_{|\sigma_i|}$ is uniquely determined.
\end{proof}

\section{Worpitzky identity for  Coxeter groups of type D} \label{section type D}

The following generalization of Worpitzky identity for type $D$ is due to Brenti \cite[Coro. 4.11]{Brenti}:
\begin{prop}\label{worp type D}
For $n \geq 2$, we have:
\begin{equation}\label{equation for type D}
(1+2x)^n-2^{n-1}(\mathcalboondox{B}_n(x+1)-\mathcalboondox{B}(n)))=\sum\limits_{k=0}^n{{n+x-k}\choose {n}}D_{n,k},
\end{equation}
where $\mathcalboondox{B}_n(\cdot)$ is the $n^{\rm th}$ Bernoulli polynomial and $\mathcalboondox{B}_n$ is the $n^{\rm th}$ Bernoulli number.
\end{prop}

By \cite[Equation (5.12)]{Mezo}, Equation (\ref{equation for type D}) above can also be written as follows:
$$(1+2m)^n-2^{n-1}(n(1^{n-1}+\cdots+m^{n-1}))=\sum\limits_{k=0}^n{{n+m-k}\choose {n}}D_{n,k}.$$

Brenti \cite{Brenti} also alludes to the following $q$-analogue:

\begin{thm}\label{worp type D with q}
For $n \geq 2$, we have:
{\tiny\begin{equation}\label{equation for type D with q}
 (1+2m)((1+q)m)^{n-1}-(1+q)^{n-1}(n(1^{n-1}+\cdots+m^{n-1}))=\sum\limits_{k=0}^n{{n+m-k}\choose {n}}D_{n,k}(q).
\end{equation}}
\end{thm}

Before proving Theorems \ref{worp type D} and \ref{worp type D with q} combinatorially, we  describe an algorithm which for a given vector $\vec{v}\in (\{0,\pm1, \dots, \pm m\})^n$ either associates with it a $D_n$-permutation or decides not to associate it with any $D_n$-permutation. When a $D_n$-permutation $\sigma$ is associated to $\vec{v}$, it should satisfy the following condition, similar to the corresponding Condition (\ref{des2 condition}) above for $B_n$:
\begin{equation}\label{desD condition}
\mbox{If } j \in {\rm Des}_D(\sigma) \mbox{, then } |a_{\sigma_j}| < |a_{\sigma_{j+1}}|
\end{equation}
(while assuming again $\sigma_0=0$ and $a_{\sigma_0}=a_0=0$).

\begin{algo}\label{D_n algorithm}

Let $\vec{v}=(a_1,\dots,a_n) \in ([-m,m])^n$ and let $\sigma \in B_n$ be the permutation associated to $\vec{a}$ by Algorithm \ref{B_n algorithm} above.

We distinguish between two cases, depending on whether or not the value $\bf{0}$ appears in $\vec{v}$.

\medskip

\noindent
{\bf First case:} The number $\bf{0}$ appears in $\vec{v}$. Let $i$ be the smallest index satisfying $a_i=\bf{0}$, and therefore $\sigma_1=i$.  We consider two sub-cases:
\begin{itemize}
\item[(a)] If the number of negative signs in $\vec{v}$ is even, then $\sigma \in D_n$.

\begin{itemize}
\item[$\bullet$] If $0 \notin {\rm Des}_D(\sigma)$, i.e. $\sigma_1+\sigma_2>0$, then we associate $\sigma$ to $\vec{v}$.
\item[$\bullet$] Otherwise, if $0 \in {\rm Des}_D(\sigma)$, then we do not associate any $D_n$-permutation to $\vec{v}$, since by Condition (\ref{desD condition}), we have $a_{\sigma_1} > a_{\sigma_0}=\bf{0}$, and so $\bf{0}$ cannot appear in $\vec{v}$ .
\end{itemize}

\medskip

\item[(b)] If the number of negative signs in $\vec{v}$ is odd (and so $\sigma \notin D_n$), then we modify $\sigma$ by inverting the sign of $\sigma_1$ (and thus considering the first appearance of $\bf{0}$ in $\vec{v}$ as negative) and denote the resulting $D_n$-permutation by $\sigma'$.

\begin{itemize}
\item[$\bullet$]
If $0 \notin {\rm Des}_D(\sigma')$, then associate $\sigma'$ to $\vec{v}$.

\item[$\bullet$] Otherwise, if $0 \in \Des_D(\sigma')$, then we do not associate any $D_n$-permutation to $\vec{v}$, again, in order to prevent a contradiction with Condition (\ref{desD condition}).
\end{itemize}

\medskip

\begin{example}
Given $n=3,m=2$ and $\vec{v}=(-2, 0, 0)$. Then  $\sigma = [2 ,3 ,-1] \notin D_3$, and so we invert the sign of $\sigma_1$ to obtain $\sigma' =[-2,3,-1] \in D_3$. The $D_3$-permutation $\sigma'$ will be associated with $\vec{v}$, since $0\notin {\rm Des}_D(\sigma')$.

On the other hand, if we take $\vec{v}=(2,0,-1)$, then we get $\sigma=[2,-3,1]\notin D_3$, so we must set $\sigma'=[-2,-3,1]\in D_3$. Since $0\in {\rm Des}_{D}(\sigma')$, we refrain from associating $\vec{v}$ with a $D_3$-permutation.
\end{example}

\end{itemize}

\medskip

\noindent
{\bf Second case:} The value $\bf{0}$ does not appear in $\vec{v}$. In this case:

\begin{itemize}
\item[$\bullet$]
If the number of negative signs in $\vec{v}$ is even, then we associate $\sigma \in D_n$ to $\vec{v}$;

\item[$\bullet$]
Otherwise, we do not associate any $D_n$-permutation to $\vec{v}$, since the obtained permutation $\sigma$ is not in $D_n$.
\end{itemize}

\end{algo}

\medskip

We summarize Algorithm \ref{D_n algorithm} in the following flowchart:

\medskip

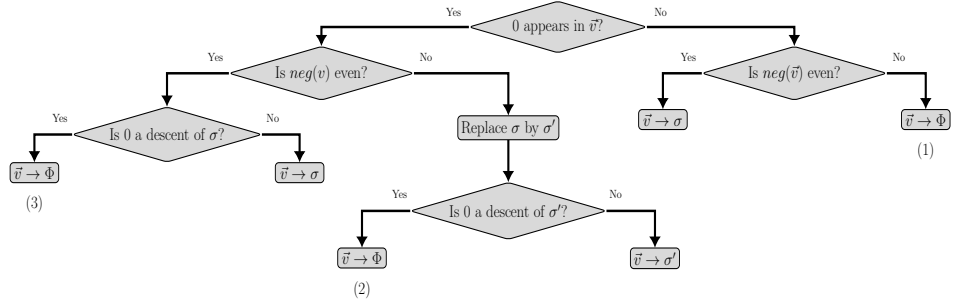
\begin{figure}[!ht]
\begin{center}
\resizebox{12.5cm}{4cm}
{\begin{tikzpicture}[node distance = 7mm and -3mm,
every path/.style = {draw, -{Latex[length=4mm, width=5mm]}}]
\node (a) [font=\Large,draw=black, diamond, rounded corners, aspect=4, rounded corners, fill=gray!30, minimum width=2cm,minimum height=0.5cm, align=center]{$0$ appears in $\vec{v}$?};

\node (left-a)[above left=-2mm and 20mm of a]{Yes};
\node (right-a)[above right=-2mm and 20mm of a]{No};

\node (b) [font=\Large,draw=black, diamond, rounded corners, aspect=4, fill=gray!30, minimum width=1.5cm,minimum height=0.4cm, align=center,below left=7mm and 60mm of a] {Is $neg(v)$ even?};

\node (left-b)[above left=-2mm and 20mm of b]{Yes};
\node (right-b)[above right=-2mm and 20mm of b]{No};

\node (c) [font=\Large,draw=black, diamond, rounded corners, aspect=4, rounded corners, fill=gray!30, minimum width=2cm,minimum height=0.5cm, align=center,below left=12mm and 25mm of b] {Is $0$ a descent of $\sigma$?};

\node (left-c)[above left=-2mm and 20mm of c]{Yes};
\node (right-b)[above right=-2mm and 20mm of c]{No};

\node (d) [font=\Large, draw=black, rounded corners, fill=gray!30, minimum width=2cm,minimum height=0.5cm, align=center,below left=5mm and 25mm of c]{$\vec{v} \rightarrow  \Phi$};

\node (left-d)[ below=3mm of d]{\Large (3)};

\node (e) [font=\Large,draw=black, rounded corners, fill=gray!30, minimum width=2cm,minimum height=0.5cm, align=center,below right=5mm and 25mm of c]{$\vec{v} \rightarrow \sigma$};

\node (f) [font=\Large,draw=black, rounded corners, fill=gray!30, minimum width=2cm,minimum height=0.5cm, align=center,below right=15mm and 18mm of b.east]{Replace $\sigma$ by $\sigma'$};

\node (g) [font=\Large,draw=black, diamond, rounded corners, aspect=4, fill=gray!30, minimum width=2cm,minimum height=0.5cm, align=center,below=15mm of f]{Is $0$ a descent of $\sigma'$?};

\node (left-g)[above left=-2mm and 20mm of g]{Yes};
\node (right-g)[above right=-2mm and 20mm of g]{No};

\node (i) [font=\Large,draw=black, rounded corners, fill=gray!30, minimum width=2cm,minimum height=0.5cm, align=center,below left=8mm and 30mm of g]{$\vec{v} \rightarrow \Phi$};

\node (left-i)[ below=3mm of i]{\Large (2)};

\node (j) [font=\Large,draw=black, diamond, rounded corners, aspect=4, rounded corners, fill=gray!30, minimum width=2cm,minimum height=0.5cm, align=center,below right=7mm and 60mm of a]{Is $neg(\vec{v})$ even?};

\node (left-j)[above left=-2mm and 20mm of j]{Yes};
\node (right-j)[above right=-2mm and 20mm of j]{No};

\node (h) [font=\Large,draw=black, rounded corners, fill=gray!30, minimum width=2cm,minimum height=0.5cm, align=center,below right=8mm and 25mm of j]{$\vec{v} \rightarrow \Phi$};

\node (left-h)[ below=3mm of h]{\Large (1)};

\node (h2) [font=\Large,draw=black, rounded corners, fill=gray!30, minimum width=2cm,minimum height=0.5cm, align=center,below left=8mm and 25mm of j]{$\vec{v} \rightarrow \sigma$};
\node (h1) [font=\Large,draw=black, rounded corners, fill=gray!30, minimum width=2cm,minimum height=0.5cm, align=center,below right=8mm and 30mm of g]{$\vec{v} \rightarrow \sigma'$};

\draw[line width=2.5pt]   (a) -| (b);
\draw[line width=2.5pt]   (b) -| (c);
\draw[line width=2.5pt]   (c) -| (d);
\draw[line width=2.5pt]   (c) -| (e);
\draw[line width=2.5pt]   (b) -| (f);
\draw[line width=2.5pt]   (a) -| (j);
\draw[line width=2.5pt]   (j) -| (h);
\draw[line width=2.5pt]   (j) -| (h2);
\draw[line width=2.5pt]   (f) -- (g);
\draw[line width=2.5pt]   (g) -| (h1);
\draw[line width=2.5pt]   (g) -| (i);
    \end{tikzpicture}}
\end{center}
\caption{Flowchart of Algorithm \ref{D_n algorithm}}\label{flowchart}
\end{figure}

\medskip

Denote by $\psi_{n,m}$ the (partial) mapping that associates a vector $\vec{v}$ with its permutation $\sigma \in D_n$ according to Algorithm \ref{D_n algorithm}.

We point out that any vector $\vec{v}$ not associated to a $D_n$-permutation by this algorithm must contain at most one zero, since if it contains more than one zero, then Algorithm \ref{B_n algorithm} (for $B_n$) yields a permutation $\sigma$ with $\sigma_2 >\sigma_1 >0$ (since $\bf{0}$'s are read from left to right in that algorithm).  Hence even if the sign of $\sigma_1$ is changed, we still have $0 \notin {\rm Des}_D(\sigma)$ and so $\vec{v}$ will be associated to some permutation ($\sigma$ or $\sigma'$) in $D_n$.

\medskip

The proof of Theorem \ref{worp type D} consists of the following two lemmas:
\blem\label{pre image for D}
For each $\sigma \in D_n$, we have $|\psi_{n,m}^{-1}(\sigma)|={ {n+m-{\rm des}_D(\sigma)}\choose{n}}$.
\elem

\blem \label{missing vectors}
The number of vectors not associated to any $D_n$-permutation by Algorithm \ref{D_n algorithm} (called \emph{'missing' vectors}) is $$2^{n-1}n\sum\limits_{j=0}^{m-1} (j+1)^{n-1}.$$
 \elem

These two lemmas together clearly prove Theorem \ref{worp type D}.
\begin{proof}[Proof of Lemma \ref{pre image for D}]
As in the proof of Theorem $\ref{worp type B}$ for type $B$, we need to show that the number of vectors associated by $\psi_{n,m}$ to each $D_n$-permutation $\sigma$ is equal to $\tiny{\left(\hspace{-4pt}\begin{array}{c}n+m-\des_D(\sigma) \\ n \end{array}\hspace{-4pt}\right)}$.
The proof is identical to the parallel
proof for type $B$, using the principle in Equation (\ref{desD condition}) above, so that if $0\in {\rm Des}_D(\sigma)$ then $|a_{\sigma_1}|> 0$ (recall that  $0 \in {\rm Des}_D(\sigma)$ if
$\sigma_1 +\sigma_2<0$).
Note that $\bf{0}$ can be considered as a negative value for the associated $D_n$-permutation, provided that $0 \notin {\rm Des}_D(\sigma)$, as in sub-case (b) of the first case in the algorithm.
\end{proof}

\begin{example} \ \\
(a) Let $\sigma=[2,-3,1,4,-5]\in D_5$ and assume that $m=4$.  Note that ${\rm Des}_D(\sigma)=\{0,1,4\}$, hence, by Condition  (\ref{desD condition}) we have\break $|a_{\sigma_1}| > 0$, so that the value ${\bf 0}$ does not appear in any vector associated to $\sigma$.

We have to find the set of vectors
$$\vec{v}=(a_1,a_2,a_3,a_4,a_5) \in (\{{\bf 0}, \pm 1 , \pm 2 ,\pm3, \pm4 \})^5,$$
satisfying $a_3<0, a_5<0$ and $a_1 \geq 0,a_2 \geq 0,a_4 \geq 0$ (in the usual integer order) and \begin{equation}\label {abs values D}
0 < |a_2| < |a_3| \leq |a_1| \leq |a_4| < |a_5| \leq 4,   \end{equation}
and one should find ${{5+4-3}\choose{5}} =6$ such vectors.

The sequence of inequalities (\ref{abs values D}) is equivalent in turn to
\begin{equation}
1 \leq \underbrace{|a_2|}_{=b_1} < \underbrace{|a_3|}_{=b_2} < \underbrace{|a_1|+1}_{=b_3} < \underbrace{|a_4|+2}_{=b_4} < \underbrace{|a_5|+2}_{=b_5} \leq 6,
\end{equation}
so we can conclude that the number of vectors satisfying the sequence of inequalities (\ref{abs values D}) is ${6 \choose 5}=6$
as needed. One can check that these vectors are: $(2,1,-2,2,-3)$, $(2,1,-2,2,-4)$, $(2,1,-2,3,-4)$,
$(3,1,-2,3,-4)$, $(3,1,-3,3,-4)$ and $(3,2,-3,3,-4)$.

\medskip

\noindent
(b) We present also an example in which $0\notin {\rm Des}_D(\sigma)$, so that $\sigma$ may be associated with vectors which contain the value ${\bf 0}$.
Let $m=2$ and let $\sigma=[-1,2,-3]\in D_3$. The requirements here are that $a_1\leq 0$ and $a_3\leq 0$ and also that
$0 \leq |a_1| \leq  |a_2| < |a_3| \leq 2$.

\noindent The vectors associated with $[-1,2,-3]$ are therefore $({\bf 0},{\bf 0},-1),({\bf 0},{\bf 0},-2)$, $({\bf 0},1,-2)$ and $(-1,1,-2)$. Note that the initial value ${\bf 0}$ in the first  three vectors are considered as negative (these vectors undergo the modification from $[1,2,-3]$ to $[-1,2,-3]$ in the first case of Algorithm \ref{D_n algorithm}).
\end{example}

\begin{proof}[Proof of Lemma \ref{missing vectors}]
There are three types of 'missing' vectors:
\begin{enumerate}
    \item {\bf Vectors which do not contain 0 and having an odd number of negative signs:} In this case, no correction of the number of signs is possible due to the lack of $\bf{0}$ (the presence of which could be used to add one to the number of negative signs), so this type of vectors is missing (Leaf (1) in the flowchart appearing in Figure \ref{flowchart}).

\medskip

\item Vectors which contain $\bf{0}$ (i.e. $a_{\sigma_1}=\bf{0}$) and have {\bf an odd number of negative signs}, such that after the modification (of $\sigma$ to $\sigma'$) we get $0\in {\rm Des}_D(\sigma')$ (Leaf (2) in the flowchart appearing in Figure \ref{flowchart}):
In this case, we must have (after the modification) $\sigma'_1<0$, and the condition $\sigma'_1+\sigma'_2<0$ implies one of the following two possibilities:

\medskip

\begin{enumerate}
        \item $|\sigma_1|>|\sigma_2|$, i.e., {\bf $\bf{0}$ is to the right of the element of $\vec{v}$ which follows $\bf{0}$ in the defined order.} In this sub-case, the sign of $\sigma_2$ is {\bf arbitrary. }

\medskip

        \item $|\sigma_1|<|\sigma_2|$, but $\sigma_2<0$.  In this sub-case, {\bf the element of $\vec{v}$ following $\bf{0}$ in the defined order must be negative and must be located to the right of $\bf{0}$ in $\vec{v}$.}
    \end{enumerate}

\medskip

\item Vectors which contain $\bf{0}$ (again, recall that this means $a_{\sigma_1}=\bf{0}$),  have {\bf an even number of negative signs}, and their associated $D_n$-permutation $\sigma$ satisfies $0 \in {\rm Des}_D(\sigma)$ (Leaf (3) in the flowchart appearing in Figure \ref{flowchart}):
In this case, we have $\sigma_1>0$, and since the vector has an even number of negative signs, the sign of $\sigma_1$ has not been changed by the algorithm. 
Combining with the fact that $0 \in {\rm Des}_D(\sigma)$, we conclude that $\sigma_2<0$ and $|\sigma_2|>\sigma_1$. These two requirements mean that {\bf the element of $\vec{v}$ following $\bf{0}$ in the defined order must be negative and must be located to the right of $\bf{0}$.}
\end{enumerate}

Before counting the 'missing' vectors, we show that none of them is already counted in the pre-image of $\psi_{n,m}$ in Lemma \ref{pre image for D} for any $D_n$-permutation.
Indeed, if $\sigma \in D_n$ satisfies that $\vec{v}\in \psi_{n,m}^{-1}(\sigma)$ and the value $\bf{0}$ does not appear in $\vec{v}$, then the number of negative elements in $\sigma$ equals the number of negative elements in $\vec{v}$. On the other hand, if the value $\bf{0}$ does appear in $\vec{v}$, but $0 \in {\rm Des}_D(\pi)$, then by Condition (\ref{desD condition}) $|a_{\sigma_1}|>0$ which contradicts the fact that $\bf{0}$ appears in $\vec{v}$.

\medskip

We proceed by counting the above 'missing' vectors.  We will have occasion throughout to refer to the absolute value of the element of $\vec{v}$ which is smallest in the defined order after $\bf{0}$, which we will call the ``second-smallest element" of $\vec{v}$.

\medskip

We start by counting the vectors appearing in the second case ($2a$), where $\bf{0}$ {\bf is to the right of the second-smallest element} and  the sign of $\sigma_2$ is {\bf arbitrary}.

Note that an element having the same absolute value as the second-smallest element but negative cannot appear to the right of $\bf{0}$.
Indeed, if $\sigma_2>0$ this will contradict the fact that we read two elements with the same absolute value but different signs starting with the negative one. On the other hand, if $\sigma_2<0$, this contradicts the fact that two negative elements are read from the right to the left.

\medskip

Our count has the form of a triple sum.  We first choose the absolute value of the second-smallest element $j$, which is located to the left of $\bf{0}$ (the outer sum), then we choose the location $i$ of $\bf{0}$, counted from the right (the middle sum, see Figure \ref{exam_case2a}), and then we choose the number of appearances $k$ of the absolute value of the second-smallest element appearing to the right of $\bf{0}$ (the inner sum). The multiplied terms in the sum are as follows:
\begin{itemize}
\item The term $\left[ (m-j+1)^{n-i}-(m-j)^{n-i}\right]$ is the number of ways to fill the places to the left of the $\bf{0}$ such that the value $j$ will appear at least once among them.

\item The term $2^{n-k-2}$ counts the number of ways to sign the $n-k-1$ elements, which are not $\bf{0}$ and are not among the $k$ positive appearances of the second-smallest element which are located to the right of $\bf{0}$, with the additional requirement of Case (2) that the total number of signs is odd.

\item The term ${i-1 \choose k}$ is the number of ways to choose the $k$ positions, out of the first $i-1$, which are to the right of $\bf{0}$, to be occupied by the positive appearances of the second-smallest element.

\item The term $(m-j)^{i-1-k}$ counts the number of ways to fill the $i-1-k$ remaining
places to the right of $\bf{0}$ with elements of larger absolute value than the second-smallest value.
\end{itemize}

\begin{figure}[!ht]
\begin{center}
{\begin{tikzpicture}

\node (a){$0$};

\node (b) at (-0.8,0) {X};

\node at (1.6,-1) {$i-1$};

\draw[line width=2pt]   (2.1,-0.3) -- (2.7,-0.3);

\draw[line width=2pt]   (1.3,-0.3) -- (1.9,-0.3);

\draw[line width=2pt]   (0.5,-0.3) -- (1.1,-0.3);

\draw[line width=2pt]   (-0.3,-0.3) -- (0.3,-0.3);

\draw[line width=2pt]   (-1.1,-0.3) -- (-0.5,-0.3);

\draw[line width=2pt]   (-1.9,-0.3) -- (-1.3,-0.3);

\draw[line width=2pt]   (-2.7,-0.3) -- (-2.1,-0.3);

\draw[line width=2pt]   (-3.5,-0.3) -- (-2.9,-0.3);

\draw [
    very thick,
    decoration={
        brace,
        mirror,
        raise=0.5cm
    },
    decorate
] (0.5,-0.1) -- (2.7,-0.1);

    \end{tikzpicture}}
\end{center}
\caption{}\label{exam_case2a}
\end{figure}

$$\sum\limits_{j=1}^m\sum\limits_{i=1}^{n-1}\sum\limits_{k=0}^{i-1} 2^{n-k-2}\left[ (m-j+1)^{n-i}-(m-j)^{n-i}\right]{i-1 \choose k} (m-j)^{i-1-k}=$$
$$\stackrel{(j \leftarrow m-j)}{=}\sum\limits_{j=0}^{m-1}\sum\limits_{i=1}^{n-1}\sum\limits_{k=0}^{i-1} 2^{n-k-2}\left[ (j+1)^{n-i}-j^{n-i}\right]{i-1 \choose k} j^{i-1-k}=$$
$$\stackrel{(binom)}{=}\sum\limits_{j=0}^{m-1}\sum\limits_{i=1}^{n-1}\left[ (j+1)^{n-i}-j^{n-i} \right]\cdot 2^{n-2}
\left(j+\frac{1}{2}\right)^{i-1}=$$
$$=\frac{1}{2}\sum\limits_{j=0}^{m-1}\sum\limits_{i=1}^{n-1}\left[ (2j+2)^{n-i}-(2j)^{n-i} \right]
(2j+1)^{i-1}=$$
$$=\frac{1}{2}\sum\limits_{j=0}^{m-1}\sum\limits_{i=1}^{n-1}\left[ (2j+2)^{n-i}(2j+1)^{i-1}-(2j)^{n-i}(2j+1)^{i-1} \right] =$$
{\tiny$$=\frac{1}{2}\sum\limits_{j=0}^{m-1}\left(\underbrace{- \left[(2j+1)^{n-1}-(2j+1)^{n-1}\right]}_{i=n} +\sum\limits_{i=1}^n\left[ (2j+2)^{n-i}(2j+1)^{i-1}-(2j)^{n-i}(2j+1)^{i-1} \right] \right) =$$}
$$=\frac{1}{2}\sum\limits_{j=0}^{m-1}\left(\sum\limits_{i=1}^n (2j+2)^{n-i}(2j+1)^{i-1}-
\sum\limits_{i=1}^n (2j)^{n-i}(2j+1)^{i-1} \right) =$$
$$\stackrel{(*)}{=}\frac{1}{2}\sum\limits_{j=0}^{m-1}\left(\frac{(2j+2)^n-(2j+1)^n}{(2j+2)-(2j+1)}-
\frac{(2j+1)^n-(2j)^n}{(2j+1)-(2j)} \right) =$$
$$=\frac{1}{2}\sum\limits_{j=0}^{m-1}\left((2j+2)^n-2(2j+1)^n+(2j)^n \right)=A,$$
where in $(*)$ we used the short multiplication formula:
$\frac{a^n-b^n}{a-b}=\sum\limits_{i=1}^n a^{n-i}b^{i-1}$.

\medskip

Next, we concentrate on Cases ($2b$)  and ($3$) together: in both cases, $\bf{0}$ {\bf is to the left of the second-smallest element, which is negative}. Since in Case ($2b$) the total number of negative signs is odd, while in Case ($3$) the total number of negative signs is even, in considering both cases together we may assume that the total number of negative signs is arbitrary.

As in the previous part,  our count of these two cases has the form of a triple sum. We first choose the second-smallest element $j$, which is located to the right of $\bf{0}$ (the outer sum), then we choose the location $i$ of $\bf{0}$, counted from the left (the middle sum, see Figure \ref{exam_case2b}), and then the number of appearances $k$ of the second-smallest element to the right of $\bf{0}$ (the inner sum), where its appearances to the left of $\bf{0}$ are counted in the terms of the sum. The multiplied terms in the sum are as follows:
\begin{itemize}
\item The term ${n-i \choose k}$ is the number of ways to choose the $k$ places out of the $n-i$ places to the right of $\bf{0}$, occupied by the elements having the same absolute value as the second-smallest element, including itself.
\item The term $2^k-1$ counts the number of ways to sign the $k$ elements having
the same absolute value as the second-smallest element appearing to the right of $\bf{0}$, excluding the unique possibility to sign all the elements as positive, since at least one of the appearances of the second-smallest element to the right of $\bf{0}$ should be negative.

\item The term $(2m-2j)^{n-i-k}$ is the number of ways to fill the places to the {\it right} of $\bf{0}$ with elements of larger absolute value than the second-smallest and with arbitrary signs.

\item The term $(2m-2j+2)^{i-1}$ counts the number of ways to fill the $i-1$ remaining
places to the {\it left} of $\bf{0}$ with elements of larger {\it or equal} absolute value to the second-smallest, and with arbitrary signs.
\end{itemize}

\begin{figure}[!ht]
\begin{center}
{\begin{tikzpicture}

\node (a){X};

\node (b) at (-0.8,0) {$0$};

\node at (-2.4,-1) {$i-1$};
\node at (1.2,-1) {$n-i$};

\draw[line width=2pt]   (2.1,-0.3) -- (2.7,-0.3);

\draw[line width=2pt]   (1.3,-0.3) -- (1.9,-0.3);

\draw[line width=2pt]   (0.5,-0.3) -- (1.1,-0.3);

\draw[line width=2pt]   (-0.3,-0.3) -- (0.3,-0.3);

\draw[line width=2pt]   (-1.1,-0.3) -- (-0.5,-0.3);

\draw[line width=2pt]   (-1.9,-0.3) -- (-1.3,-0.3);

\draw[line width=2pt]   (-2.7,-0.3) -- (-2.1,-0.3);

\draw[line width=2pt]   (-3.5,-0.3) -- (-2.9,-0.3);

\draw [
    very thick,
    decoration={
        brace,
        mirror,
        raise=0.5cm
    },
    decorate
] (-3.5,-0.1) -- (-1.3,-0.1);

\draw [
    very thick,
    decoration={
        brace,
        mirror,
        raise=0.5cm
    },
    decorate
] (-0.3,-0.1) -- (2.7,-0.1);

    \end{tikzpicture}}
\end{center}
\caption{}\label{exam_case2b}
\end{figure}

$$\sum\limits_{j=1}^m\sum\limits_{i=1}^{n-1}\sum\limits_{k=1}^{n-i} (2^k-1){n-i \choose k} (2m-2j)^{n-i-k} (2m-2j+2)^{i-1}=
$$
$$\stackrel{(j \leftarrow m-j)}{=} \sum\limits_{j=0}^{m-1}\sum\limits_{i=1}^{n-1}\sum\limits_{k=1}^{n-i} (2^k-1){n-i \choose k} (2j)^{n-i-k} (2j+2)^{i-1}=$$
{\tiny $$=\sum\limits_{j=0}^{m-1}\sum\limits_{i=1}^{n-1}(2j+2)^{i-1}\sum\limits_{k=1}^{n-i} 2^k{n-i \choose k} (2j)^{n-i-k}-
\sum\limits_{j=0}^{m-1}\sum\limits_{i=1}^{n-1}(2j+2)^{i-1}\sum\limits_{k=1}^{n-i} {n-i \choose k} (2j)^{n-i-k} =
$$
$$\stackrel{(binom)}{=}\sum\limits_{j=0}^{m-1}\sum\limits_{i=1}^{n-1}(2j+2)^{i-1} [(2j+2)^{n-i} -\underbrace{(2j)^{n-i}}_{k=0} ] \ \  -
\sum\limits_{j=0}^{m-1}\sum\limits_{i=1}^{n-1}(2j+2)^{i-1}[(2j+1)^{n-i} -\underbrace{(2j)^{n-i}}_{k=0} ] =
$$}
$$=\sum\limits_{j=0}^{m-1}\sum\limits_{i=1}^{n-1}(2j+2)^{i-1} \left[(2j+2)^{n-i} -(2j+1)^{n-i} \right]= $$
$$=\sum\limits_{j=0}^{m-1}\sum\limits_{i=1}^{n-1}\left[(2j+2)^{n-1} -  (2j+2)^{i-1} (2j+1)^{n-i} \right] =  $$
$$=\sum\limits_{j=0}^{m-1}\left( (n-1)(2j+2)^{n-1}  - \sum\limits_{i=1}^{n-1}  (2j+2)^{i-1} (2j+1)^{n-i} \right)= $$
$$=\sum\limits_{j=0}^{m-1}\left( (n-1)(2j+2)^{n-1}  - \left(\sum\limits_{i=1}^{n}  (2j+2)^{i-1} (2j+1)^{n-i}\right) +\underbrace{(2j+2)^{n-1}}_{i=n}\right)=  $$
$$\stackrel{(*)}{=}\sum\limits_{j=0}^{m-1}\left( n(2j+2)^{n-1}-
\frac{(2j+2)^n-(2j+1)^n}{(2j+2)-(2j+1)}\right)=  $$
$$=\sum\limits_{j=0}^{m-1}\left[ n(2j+2)^{n-1} -(2j+2)^n+(2j+1)^n\right]= B,$$
where in $(*)$ we used the short multiplication formula:
$\frac{a^n-b^n}{a-b}=\sum\limits_{i=1}^n a^{n-i}b^{i-1}$.

\medskip

We now sum up together Cases (2) and (3) (counted by expressions $A$ and $B$):
{\tiny $$\underbrace{\frac{1}{2}\sum\limits_{j=0}^{m-1}\left((2j+2)^n-\underline{2(2j+1)^n}+(2j)^n \right)}_{=A}+\underbrace{\sum\limits_{j=0}^{m-1}\left[ n(2j+2)^{n-1} -(2j+2)^n+\underline{(2j+1)^n}\right]}_{=B}=$$}
$$=\sum\limits_{j=0}^{m-1}\left( \frac{1}{2}((2j+2)^n+(2j)^n)+ n(2j+2)^{n-1}-(2j+2)^n \right)=$$
$$=\sum\limits_{j=0}^{m-1}\left( n(2j+2)^{n-1}+
\frac{1}{2}((2j)^n-(2j+2)^n) \right)=$$
$$=\sum\limits_{j=0}^{m-1} 2^{n-1}n(j+1)^{n-1}+\sum\limits_{j=0}^{m-1}
\frac{1}{2}((2j)^n-(2j+2)^n) \stackrel{telescopic}{=}$$
$$=2^{n-1}n\sum\limits_{j=0}^{m-1} (j+1)^{n-1}-\frac{1}{2}(2m)^n=2^{n-1}n\sum\limits_{j=0}^{m-1} (j+1)^{n-1}-2^{n-1}m^n$$
Adding the number of the vectors of Case (1), which is clearly $2^{n-1}m^n$ (half the total number of vectors not containing $\bf{0}$), yields the total number of vectors not associated to any $D_n$-permutation, namely: $$2^{n-1}n\sum\limits_{j=0}^{m-1} (j+1)^{n-1}.$$
This completes the proof of Theorem \ref{worp type D} as well.
\end{proof}

The proof of Theorem \ref{worp type D with q} is based on Lemma \ref{pre image for D} and on the following lemma:

\begin{lem} \label{missing vectors with q}
The weight contributed by the vectors not associated to any $D_n$-permutation by Algorithm \ref{D_n algorithm} is $$(1+q)^{n-1}n\sum\limits_{j=0}^{m-1} (j+1)^{n-1}.$$
\end{lem}

\begin{proof}[Proof of Lemma \ref{missing vectors with q}]

As in Lemma \ref{missing vectors}, there are three types of 'missing' vectors. We now count their contribution with regard to  the $q$-analogue.

\begin{enumerate}
    \item[Case (1):]
We put $q$ for each negative term in $\vec{v}$ except for the smallest one and consider two cases:
\begin{enumerate}
\item If the sign of the smallest element is {\bf positive}, then we have $m$ options to choose its value and the other elements contribute $\frac{((1+q)m)^{n-1}}{2} \cdot m$, since the number of negative elements must be odd.
\item If the sign of the smallest element is {\bf negative}, then after choosing the smallest element, the number of negatives among the remaining elements is even so we have again: $\frac{((1+q)m)^{n-1}}{2} \cdot m$.
\end{enumerate}
In total we have $(1+q)^{n-1}m^n$.

\medskip

\item[Case (2a):] We have the following triple sum:
{\tiny $$\sum\limits_{j=1}^m\sum\limits_{i=1}^{n-1}\sum\limits_{k=0}^{i-1} (1+q)^{n-k-2}\left[ (m-j+1)^{n-i}-(m-j)^{n-i}\right]{i-1 \choose k} (m-j)^{i-1-k}.$$}

\medskip

\item[Cases (2b)+(3):] We have the following triple sum:
{\tiny $$\sum\limits_{j=1}^m\sum\limits_{i=1}^{n-1}\sum\limits_{k=1}^{n-i} ((1+q)^k-1){n-i \choose k} ((1+q)m-(1+q)j)^{n-i-k} ((1+q)m-(1+q)j+(1+q))^{i-1}
.$$}
\end{enumerate}

\medskip

Applying manipulations similar to the ones we used in Lemma \ref{missing vectors}, while replacing $2$ by $(1+q)$, yields the required result.
\end{proof}

\section*{Acknowledgments}

In order to overcome the messy triple sums in Lemma \ref{missing vectors}, we used both the Mathematica and Maple symbolic computation abilities. We want to thank Aharon Naiman for writing a Mathematica program and Doron Zeilberger, Noah Dana-Picard and Ilias Kotsireas for helping us with the Maple.

\end{document}